\documentclass{amsart}
\usepackage{amsmath,amsfonts,amsthm,amssymb}

\newcommand{\R}{\mathbb{R}}
\newcommand{\C}{\mathbb{C}}
\newcommand{\Z}{\mathbb{Z}}
\newcommand{\N}{\mathbb{N}}
\newcommand{\F}{\mathcal{F}}
\newcommand{\D}{\mathcal{D}}
\newcommand{\dir}{\mathcal{E}}
\newcommand{\ie}{\emph{i.$\,$e.$\!$} }
\newcommand{\e}{\varepsilon}
\newcommand{\Heat}{\mathcal{H}}
\newcommand{\Pois}{\mathcal P}
\newcommand{\la}{\lambda}

\newtheorem{theorem}{Theorem}[section]
\newtheorem{corollary}[theorem]{Corollary}
\newtheorem{lemma}[theorem]{Lemma}
\newtheorem{proposition}[theorem]{Proposition}

\theoremstyle{definition}
\newtheorem{definition}{Definition}[subsection]

\theoremstyle{remark}
\newtheorem{remark}[theorem]{Remark}
\newtheorem{step}{Step}

\DeclareMathOperator{\dom}{dom}

\numberwithin{equation}{section}

\begin{document}

\title[Fatou-type theorems on PCF sets]{Nontangential limits and Fatou-type
theorems on post-critically finite self-similar sets}

\author{Ricardo A. S\'aenz}

\address{Facultad de Ciencias\\Universidad de Colima\\Colima, Colima, Mexico,
28045}

\email{rasaenz@ucol.mx}
\urladdr{http://fejer.ucol.mx/ricardo}

\date{\today}



\begin{abstract}
In this paper we study the boundary limit properties of harmonic functions on
$\R_+\times K$, the solutions $u(t,x)$ to the Poisson equation
\[
\frac{\partial^2 u}{\partial t^2} + \Delta u = 0,
\]
where $K$ is a p.c.f. set and $\Delta$ its Laplacian given by a regular
harmonic structure. In particular, we prove the existence of nontangential
limits of the corresponding Poisson integrals, and the analogous results
of the classical Fatou theorems for bounded and nontangentially bounded
harmonic functions.

\end{abstract}

\maketitle

\section{Introduction}

There has recently been a growing interest in the study of analysis on
fractals, in particular post-critically finite (p.c.f.) self similar sets and 
their harmonic structure defined by Kigami \cite{Kigami93}. Analogous
questions from classical analysis have been asked on the setting of p.c.f.
fractals, from spectral theory of the Laplacian \cite{ASST03,CSW09,ORS10},
functional analysis \cite{Strichartz03,RogStr10,IR10} and differential 
equations \cite{DSV99,Strichartz05,Pelander07}.

In this paper we study the boundary limit properties of harmonic functions on
the tube $\R_+\times K$, the solutions $u(t,x)$ to the Poisson equation
\[
\frac{\partial^2 u}{\partial t^2} + \Delta u = 0,
\]
where $K$ is a p.c.f. set and $\Delta$ its Laplacian given by a regular
harmonic structure. In Section \ref{kernels-sect} we define the Poisson kernel
and prove its elementary properties, as well as proving the existence of
nontangential limits of Poisson integrals in the boundary $t\to 0$.
Nontangential limits are  defined in terms of proper "cones", depending on the 
Hausdorff dimension of $K$ with respect to effective resistence metric.

In Section \ref{Fatou-sect} we prove an analogous Fatou theorem for bounded
Dirichlet harmonic functions on $\R_+\times K$. We also extend these results to
Dirichlet harmonic functions with uniformly bounded $L^p$ norms,
$1\le p < \infty$.

We finish the paper with the analogous version to the local Fatou theorem, in 
Section \ref{local-sect}, for nontangentially bounded harmonic functions on 
$\R_+\times K$. As we make use of estimates from below for the Neumann heat 
kernel on $K$, we prove this results only for nested fractals \cite{Lindstrom}.

\section{Preliminaries}

\subsection{P.c.f. self-similar structures}

Let $(K,S,\{F_i\}_{i\in S})$ be a self-similar structure. $W_m=S^m$ is the set
of words of length $m$, and $W_*=\bigcup_{m\ge0}W_m$, where 
$W_0 = \{\emptyset\}$ and $\emptyset$ is called the empty word. For $w\in W_m$,
we write $F_w = F_{w_1}\circ\cdots\circ F_{w_m}$ ($F_\emptyset$ is set to be
the identity mapping) and $K_w = F_w(K)$. The critical set of $K$ is the set
\[
\mathcal C = \bigcup_{\substack{i,j\in S\\i\not=j}}K_i\cap K_j,
\]
and the post-critical set is given by 
$V_0 = \{p\in K:\exists\text{$w\in W_*$, $F_w(p)\in \mathcal C$}\}$. We say
that $K$ is a post-critically finite (p.c.f.) self-similar structure if the 
set $V_0$ is finite. From now on, we will assume that $K$ is a 
p.c.f.\/ self-similar structure, and that $V_0\not=\emptyset$.

$V_0$ will be called the boundary of $K$. We define 
$V_m = \bigcup_{w\in W_m} F_w(V_0)$, and $V_* = \bigcup_{m\ge0}V_m$. Basic
properties and examples of p.c.f. self-similar structures can be found in
\cite[Section 1.3]{Kigami}.

If $0<\mu_i<1$ and $\sum_{i\in S}\mu_i=1$, let $\mu$ be the Bernoulli measure
with weights $(\mu_i)_{i\in S}$. It satisfies 
$\mu(K_w)=\mu_w = \mu_{w_1}\cdots\mu_{w_m}$ for $w\in W_m$ and, for any
integrable function $f$ on $K$,
\[
\int_K f d\mu = \sum_{w\in W_m}\mu_w \int_K f\circ F_w d\mu.
\]

We define the Banach spaces $L^p(K,\mu)$, $1\le p\le\infty$, as usual.

\subsection{Harmonic structure}

Let  $(D,\mathbf{r})$ be a regular harmonic structure on $K$, where 
$\mathbf{r}=(r_i)_{i\in S}$ with $0 < r_i < 1,i\in S$. We denote by $R(x,y)$
the effective resistance metric induced by $(D,\mathbf{r})$. Under this metric,
$V_*$ is dense in $K$ \cite[Section 3.3]{Kigami} and the functions $F_w$
become contractions with Lipschitz constant
$\text{Lip\,} F_w\le r_w = r_{w_1}\cdots r_{w_m}$.

Let $\dir(f,g)$ be the Dirichlet form associated to $(D,\mathbf r)$, defined on
$\F\subset C(K)$. It satisfies, for $m\ge 1$,
\begin{equation*}
\dir(f,g) = \sum_{w\in W_m}\frac{1}{r_w}\dir(f\circ F_w,g\circ F_w).
\end{equation*}

By a theorem of Kumagai \cite{Kumagai93} (see \cite[Section 3.4]{Kigami}), for
any self-similar measure $\mu$ on $K$, $(\dir,\F)$ is a local regular Dirichlet
form on $L^2(K,\mu)$, and the corresponding non-negative self-adjoint
operator $H_N$ on $L^2(K,\mu)$ has compact resolvent. If we define
$\F_0=\{u\in\F: u|_{V_0} = 0\}$, then $(\dir,\F_0)$ is a local Dirichlet form
on $L^2(K,\mu)$, and the corresponding operator $H_D$ has also compact
resolvent. Moreover, $H_D$ is invertible, and $(H_D)^{-1}$ is a compact
operator on $L^2(K,\mu)$. The operators $-H_N$ and $-H_D$ are called the
\emph{Neumann} and \emph{Dirichlet} Laplacians, respectively.

\subsection{Maximal function}\label{Maximal}

Let $d$ be the unique real number that satisfies $\sum_{i\in S}(r_i)^d = 1$.
$d$ is called the \emph{similarity dimension} of the harmonic structure 
$(D,\mathbf{r})$, and it can be proved that $d$ corresponds to the Hausdorff
dimension of $K$ with respect to the resistance metric $R$
\cite[Section 4.2]{Kigami}. If $\mu$ is the self-similar measure on $K$ with
weights $\mu_i = (r_i)^d$, $i\in S$, it is easy to prove that
$\mu\big(B_\e(x)\big) \sim \e^d$ for any $\e>0$ sufficiently small. More
precisely, there exist two constants $A_1, A_2>0$ such that, for any $x\in K$
and sufficiently small $\e>0$,
\begin{equation}\label{scalmu}
A_1\e^d \le \mu(B_\e(x))\le A_2 \e^d.
\end{equation}
See \cite{Saenz02} for details. For $f\in L^1(K,\mu)$, we define the operator
\begin{equation}\label{maxfunction}
Mf(x) = \sup_{\e>0}\frac{1}{\mu\big(B_\e(x)\big)}\int_{B_\e(x)}|f|d\mu,
\end{equation}
where $B_\e(x)$ is the ball of radius $\e$ around $x$ with respect to the 
effective resistance metric. The following theorem is standard \cite[Chapter
III]{Stein70}.

\begin{theorem}\label{maxthm}
If $f$ is an integrable function and $Mf$ is given by \eqref{maxfunction},
then $Mf(x)$ is finite a.e. Moreover,
\begin{enumerate}
\item There exists a constant $A>0$ such that, for any $f\in L^1(K,\mu)$
and $\alpha>0$,
\[
\mu\big(\{x\in K: Mf(x)>\alpha\}\big) \le \frac{A}{\alpha}||f||_{L^1(K,\mu)};
\]
\item $M$ extends to a bounded operator on $L^p(K,\mu)$ for $1<p\le\infty$.
\end{enumerate}
\end{theorem}

\begin{remark}
In the case where $K$ is the interval $I=[0,1]$, the functions $F_1$ and $F_2$
are the contractions $x\mapsto x/2$ and $x\mapsto x/2+1/2$, and the harmonic
structure corresponds to
$D=(\begin{smallmatrix}-1&1\\1&-1\end{smallmatrix})$ and $r_1=r_2=1/2$, we have
that the effective resistance metric $R$ is equal to the standard metric and
$\mu$ is the Lebesgue measure in $I$. In such case, \eqref{maxfunction} is the 
Hardy-Littlewood maximal operator and Theorem \ref{maxthm} is the classical 
maximal function theorem. 
\end{remark}

\begin{remark}
It is not hard to see that, as in the classical case, Theorem \ref{maxthm} is
also true for finite Borel measures on $K$. Precisely, if $\nu$ is a finite
Borel measure and we define the function $M\nu$ as
\[
M\nu(x) = \sup_{\e>0}\frac{1}{\mu\big(B_\e(x)\big)}\int_{B_\e(x)}|d\nu|,
\]
then 
\[
\mu\big(\{x\in K: M\nu(x)>\alpha\}\big) \le\frac{A}{\alpha}||\nu||.
\]
(Cf. \cite[Chapter III, 4.1]{Stein70}.)
\end{remark}

\subsection{Laplacians}\label{Laplacians}

Let $\Delta$ be the Laplacian associated with $(D,\mathbf r)$ and
$\mu$. We denote its domain by $\D$. We also consider the sets 
\[ \D_D = \{u\in\D:u|_{V_0}=0\}\quad\text{ and }\quad
\D_N =  \{ u \in \D: du = 0 \text{ on } V_0\},
\]
where $df(p)$ is the Neumann derivative of $f$ at the boundary point $p$
\cite{Kigami}. One can see then that $H_D$ is the Friedrich extension of
$-\Delta$ on $\D_D$, while $H_N$ is the Friedrich extension of $-\Delta$ on
$\D_N$ (see \cite[Section 3.7]{Kigami}).

For convenience, as in \cite{Kigami}, will denote by $b$ either $D$ or $N$;
so, for instance, $H_b$ will denote the operator $H_D$ or $H_N$, respectively.

\subsection{Dirichlet and Neumann eigenfunctions}\label{Dir-eigen}

Consider the set $E_b(\lambda) = \{\phi\in\D_b: \Delta \phi = -\lambda\phi\}$.
If $\dim E_D(\lambda) \not=0$, $\lambda$ is called a Dirichlet eigenvalue, and
the collection of such $\lambda$ is called the Dirichlet spectrum of $\Delta$.
If $\dim E_N(\lambda) \not= 0$, then $\lambda$ is called a Neumann eigenvalue,
and the collection of such $\lambda$ is called the Neumann spectrum of
$\Delta$. Both the Dirichlet and Neumann spectra of $\Delta$ are subsets of
$[0,\infty)$, and there are $\lambda_n^b$ and $\phi^b_n\in E_b(\la_n^b)$ such
that
\[
0 \le \lambda_1^b \le \lambda_2^b \le \ldots
\]
and $\{\phi_n^b:n\ge1\}$ is a complete orthonormal system for $L^2(K,\mu)$.
Observe that $\lambda^D_1 > 0$. The following theorem is due to Kigami and
Lapidus \cite{KL93}.

\begin{theorem}\label{spectral}
If $\rho_b(x) = \sum_{\la\le x}\dim E_b(\la)$,
\begin{equation}\label{weyl}
0 < \liminf_{x\to\infty}\frac{\rho_b(x)}{x^{d/(d+1)}} \le
\limsup_{x\to\infty}\frac{\rho_b(x)}{x^{d/(d+1)}} < \infty.
\end{equation}
\end{theorem}
Equation \eqref{weyl} is the analogous to Weyl's formula which counts the
eigenvalues of the Laplacian on a domain in $\R^n$. We will also make use of
the following property:
\begin{enumerate}
\item[*] There exists a constant $C>0$ such that
\begin{equation}\label{est}
||\phi||_\infty \le C\la^\frac{d}{2(d+1)} ||\phi||_2,\qquad\phi\in E_b(\la).
\end{equation}
\end{enumerate}
A proof of this can be found in \cite[Section 4.5]{Kigami}.

\section{The heat and Poisson kernels}\label{kernels-sect}

\subsection{Heat kernel}

For $b=D$ or $N$, the Dirichlet (respectively Neumann) heat kernel is the
function $H^b:\R_+\times K\times K\to\C$ defined by
\begin{equation}\label{heat}
H^b(t,x,y) = \sum_{n=1}^\infty e^{-\lambda_n^b t}\phi_n^b(x)\phi^b_n(y).
\end{equation}
Although the right-hand sum of \eqref{heat} is defined only formally, it is not
hard to prove that it converges uniformly on $[T,\infty)\times K\times K$ for
any $T>0$, which follows from the results in Section \ref{Dir-eigen} (cf.
Section \ref{Poisson-section}). Moreover, $H^b$ is nonnegative, continuous,
defines a fundamental solution to the heat equation \cite[Section 5.1]{Kigami},
and
\begin{equation}\label{heat-sg}
\int_K H^b(t,x,z)H^b(s,z,y) d\mu(z) = H^b(t+s,x,y).
\end{equation}
These properties imply that the operator $f \mapsto \Heat_t^b f$, where
\[
\Heat^b_t f(x) = \int_K H^b(t,x,y) f(y) d\mu(y),
\]
defined for $t>0$ and integrable functions $f$ on $K$, is a strongly continuous
semigroup on $L^2(K,\mu)$ whose generator is given by $-H_b$. In fact, as
$H^b$ is continuous and $K$ is compact, we have that $\Heat^b_t$ is bounded from
$L^p(K,\mu)$ to $C(K)$ for any $t>0$ and $1\le p\le \infty$ and, moreover,
\[
\Heat^b_t(L^p(K,\mu)) \subset \D_b.
\]
If $u(t,x)$ is defined on $\R_+\times K$ as $u(t,x) = \Heat^b_t f(x)$, then
\[
\frac{\partial u(t,x)}{\partial t} = \Delta u(t,x),
\]
so $H^b$ is the fundamental solution of the heat equation on $K$. For the
details of these facts, see \cite[Chapter 5]{Kigami}. The following result is
also proved in \cite[Proposition 5.2.6]{Kigami}.

\begin{proposition}\label{identHeat}
\begin{enumerate}
\item For $f\in C(K)$, $||\Heat^N_t f - f||_{L^\infty(K,d\mu)} \to 0$ as 
$t\to 0$.
\item Let $f\in C(K)$, with $f|_{V_0} \equiv 0$. Then 
$||\Heat^D_t f - f||_{L^\infty(K,d\mu)} \to 0$ as $t\to 0$.
\end{enumerate}

\end{proposition}

Thus we observe that the heat kernel acts as an \emph{approximation to the
identity} for continuous functions.


\subsection{Poisson kernel}\label{Poisson-section}

We define the \emph{Dirichlet} (respectively \emph{Neumann}) \emph{Poisson
kernel} $P^b:\R_+\times K\times K\to \C$ by
\begin{equation}\label{poisson}
P^b(t,x,y) = \sum_{n=1}^\infty e^{-\sqrt{\lambda^b_n} t} 
\phi^b_n(x)\phi^b_n(y).
\end{equation}
As above, we first take the series in \eqref{poisson} formally. However, the
series converges uniformly on $[T,\infty)\times K \times K$ for any $T>0$,
which follows from the following two observations.
\begin{enumerate}
\item There exist $c_1,c_2>0$ such that
\begin{equation}\label{est-l}
c_1 n^{(d+1)/d} \le \lambda^b_n \le c_2 n^{(d+1)/d};
\end{equation}
\item For $\alpha,\beta,\gamma,T>0$, 
$\sum_{n\ge 1} n^\alpha e^{-\gamma n^\beta t}$ converges uniformly for
$t\in[T,\infty)$.
\end{enumerate}
The first follows from Theorem \ref{spectral}, while the second is
straightforward from the well-known properties of the exponential function.
Now, from these and equation \eqref{est} it follows that 
\[
|e^{-\sqrt{\lambda^b_n} t} \phi^b_n(x)\phi^b_n(y)| 
\lesssim e^{-\sqrt{c_1}tn^{(d+1)/(2d)}} n
\]
uniformly in $K\times K$, and thus the series in \eqref{poisson} converges
uniformly on $[T,\infty)\times K\times K$ for any $T>0$.

The following identity, based on the \emph{principle of subordination} 
\cite[Section III.2]{Stein70}, will be useful to study the properties of
$P^b(t,x,y)$.

\begin{proposition}\label{subord}
For $(t,x,y)\in\R_+\times K\times K$,
\[
P^b(t,x,y) = \frac{t}{2\sqrt \pi} \int_0^\infty e^{-t^2/4s} H^b(s,x,y)
\frac{ds}{s^{3/2}},
\]
where $H^b$ is the heat kernel on $K$.
\end{proposition}

\begin{proof}
By the estimates \eqref{est} and \eqref{est-l}, we see that
\[
H^b(t,x,y) \le \sum_{n=1}^\infty e^{-\lambda^b_n t} (\lambda^b_n)^{d/(d+1)} \le 
\sum_{n=1}^\infty n e^{-c_1 n^{(d+1)/d} t},
\]
uniformly on $K\times K$. Let $\alpha = \dfrac{d}{(d+1)}$. Since,
$e^{-c_1 n^{1/\alpha} t} \lesssim \dfrac{1}{(n^{1/\alpha} t)^{2\alpha + 1}}$,
we see that the series can be estimated by
\[
\sum_{n=1}^\infty n e^{-c_1 n^{1/\alpha} t} \lesssim \frac{1}{t^{2\alpha+1}}.
\]
Now the integral $\displaystyle\int_0^\infty e^{-t^2/4s}
\frac{1}{s^{2\alpha+1}}\frac{ds}{s^{3/2}}$ converges for $t>0$, so by the
dominated convergence theorem
\begin{equation*}
\begin{split}
\int_0^\infty e^{-t^2/4s} H^b(s,x,y) \frac{ds}{s^{3/2}} &= 
\sum_{n=1}^\infty \phi^b_n(x)\phi^b_n(y)
\int_0^\infty e^{-t^2/4s} e^{-\lambda^b_n s} \frac{ds}{s^{3/2}}\\ &=
\sum_{n=1}^\infty \phi^b_n(x)\phi^b_n(y)
\frac{2\sqrt{\pi}}{t} e^{-\sqrt{\lambda^b_n} t} =
\frac{2\sqrt{\pi}}{t} P^b(t,x,y),
\end{split}
\end{equation*}
where we have used the identity
\[
\int_0^\infty e^{-t^2/4s} e^{-\beta^2 s} \frac{ds}{s^{3/2}} = 
\frac{2\sqrt{\pi}}{t} e^{-\beta t}
\]
for $\beta>0$ \cite[Section III.2]{Stein70}.
\end{proof}

Proposition \ref{subord} allows us to conclude the following properties 
of $P^b(t,x,y)$, analogously to those of $H^b(t,x,y)$.

\begin{corollary}\label{PoissonProp}
The Poisson kernel satisfies the following properties.
\begin{enumerate}
\item $P^b$ is nonnegative and continuous;
\item For $(t,x)\in\R_+\times K$, $P^b(t,x,\cdot)\in\D_b$;
\item For $(x,y)\in K\times K$, $P^b(\cdot,x,y)\in C^2(\R_+)$;
\item For $(t,x,y)\in\R_+\times K \times K$, 
\begin{equation}\label{Poisson-fund}
\frac{\partial^2 P^b(t,x,y)}{\partial t^2} + (\Delta P^b(t,x,\cdot))(y) = 0;
\end{equation}
and
\item For $t,s\in\R_+$, $x,y\in K$, 
\begin{equation}\label{Poisson-sg}
\int_K P^b(t,x,z)P^b(s,z,y) d\mu(z) = P^b(t+s,x,y).
\end{equation}
\end{enumerate}
\end{corollary}

\begin{proof}
(1) follows from the nonnegativity of $H^b(t,x,y)$ and the fact that the
integral in Proposition \ref{subord} converges absolutely.

To prove (2), observe that $\sum_{n\ge 1} a_n\phi^b_n \in \dom(H_b)$ if and
only if
\[
\sum_{n\ge 1} |\lambda^b_n a_n|^2 < \infty, 
\]
which clearly holds for $a_n = e^{-\sqrt{\lambda^b_n}t}\phi^b_n(x)$
for any given $t>0$ and $x\in K$, because of \eqref{est}. Then
$P^b(t,x,\cdot)\in\dom(H_b)$ and
\[
H_bP^b(t,x,\cdot) = \sum_{n=1}^\infty \lambda^b_n e^{-\sqrt{\lambda^b_n} t}
\phi^b_n(x)\phi^b_n.
\]
As above, one can verify that this series converges uniformly on 
$[T,\infty)\times K\times K$, and thus $H_bP^b(t,x,\cdot)\in C(K)$. Therefore
$P^b(t,x,\cdot)\in\D_b$.

Now fix $x,y\in K$ and set $f_n(t) = e^{-\sqrt{\lambda^b_n} t}
\phi^b_n(x)\phi^b_n(y)$.
Since 
\[
\sum_{n=1}^\infty f_n'(t) = 
-\sum_{n=1}^\infty \sqrt{\lambda^b_n}e^{-\sqrt{\lambda^b_n} t}
\phi^b_n(x)\phi^b_n(y)
\]
converges uniformly on $[T,\infty)$ for any $T>0$, $t\mapsto P^b(t,x,y)$ is
continuosly differentiable and
\[
\frac{\partial P^b(t,x,y)}{\partial t} = \sum_{n=1}^\infty f_n'(t).
\]
Now $f_n''(t) = \lambda^b_n e^{-\sqrt{\lambda^b_n} t} \phi^b_n(x)\phi^b_n(y)$,
so
$\sum_{n=1}^\infty f_n''(t)$ also converges absolutely on $[T,\infty)$ for any
$T>0$ and thus $t\mapsto P^b(t,x,y)$ is in $C^2(\R_+)$, which proves (3).

(4) follows from the fact that
\[
\begin{split}
\frac{\partial^2 P^b(t,x,y)}{\partial t^2} &= \sum_{n=1}^\infty f_n''(t) =
\sum_{n=1}^\infty \lambda^b_n e^{-\sqrt{\lambda^b_n} t} \phi^b_n(x)\phi^b_n(y)
\\
&= (H_bP^b(t,x,\cdot))(y) = - (\Delta P^b(t,x,\cdot))(y).
\end{split}
\]

For (5), it is sufficient to note that \eqref{Poisson-sg} follows from
\eqref{heat-sg},
Proposition \ref{subord} and Fubini's theorem.
\end{proof}

\subsection{Poisson semigroup}\label{Poisson-SG-section}

The results from the previous section lead us, analogously to the heat kernel,
to define the operators $f\mapsto \Pois^b_tf$ ($b=D$ or $N$) for each $t>0$ as
\begin{equation}\label{PoissonOper}
\Pois^b_t f(x) = \int_K P^b(t,x,y) f(y) d\mu(y),
\end{equation}
defined for integrable functions $f$ on $K$. The continuity of $P^b$ and the 
compactness of $K$ imply that $\Pois^b_t$ is bounded from $L^p(K,\mu)$ to
$C(K)$ for any $t>0$ and $1\le p \le \infty$. We also see that
$\Pois^b_t \circ \Pois^b_s = \Pois^b_{t+s}$, which follows from
\eqref{Poisson-sg}, so we have that $\{\Pois^b_t\}_{t>0}$ is a semigroup.

We in fact have the following theorem.

\begin{theorem}\label{Psemi}
Let $f\in L^p(K,\mu)$ and define, for $(t,x)\in K$,
\[
u(t,x) = \Pois^b_t f(x).
\]
\begin{enumerate}
\item For each $x\in K$, $u(\cdot,x)\in C^\infty(\R_+)$;
\item For each $t>0$, $u(t,\cdot)\in\D_b$; and
\item For each $(t,x)\in\R_+\times K$,
\[
\frac{\partial^2 u(t,x)}{\partial t^2} + \Delta u(t,x) = 0.
\]
\end{enumerate}
\end{theorem}

\begin{proof}
Since, for $f\in L^p$,
\[
\int_K |\phi^b_n(y)f(y)|d\mu(y) \le C \lambda_n^\frac{d}{2(d+1)}||f||_{L^p}
\]
by \eqref{est}, we have that, for every positive integer $m$,
\[
\Big|\Big|\sum_{n=1}^m e^{-\sqrt{\lambda^b_n}t}\phi^b_n(x)\phi^b_n
f\Big|\Big|_{L^1}
\le C ||f||_{L^p} \sum_{n=1}^\infty
e^{-\sqrt{\lambda^b_n}t}(\lambda^b_n)^\frac{d}{d+1},
\]
so by \eqref{est-l} the sum is uniformly bounded in $m$ by $C' ||f||_{L^p}$.
The dominated convergence theorem implies then that 
$u(t,x) = \sum_{n=1}^\infty a_n e^{-\sqrt{\lambda^b_n}t}\phi^b_n(x)$, where
\[
a_n = \int_K \phi^b_n(y) f(y) d\mu(y).
\]

Now fix $x\in K$ and set $\psi_n(t) = a_n e^{-\sqrt{\lambda^b_n}t}\phi^b_n(x)$.
For each $k\in\N$,
\[
\psi^{(k)}(t) = a_n (\lambda^b_n)^{k/2} e^{-\sqrt{\lambda^b_n}t}\phi^b_n(x),
\]
so the series $\sum \psi_n^{(k)}(t)$ converges uniformly for $t\in[T,\infty)$,
for any $T>0$,  by \eqref{est-l}. It follows that $t\mapsto u(t,x)$ is in
$C^k(\R_+)$ for any $k$, which proves (1).

We now fix $t>0$. For (2), we first need to verify that
\[
\sum_{n=1}^\infty |a_n e^{-\sqrt{\lambda^b_n}t}|^2 < \infty.
\]
But this, again, follows from \eqref{est-l} and the fact that 
$|a_n|\le C (\lambda^b_n)^\frac{d}{2(d+1)} ||f||_{L^p}$. This shows $x\mapsto
u(t,x)$ is in $\dom(H_b)$ and
\[
H_b u(t,\cdot) = \sum_{n=1}^\infty a_n \lambda^b_n
e^{-\sqrt{\lambda^b_n}t} \phi^b_n.
\]
As above, this series converges uniformly, so we have that 
$H_b u(t,\cdot) \in C(K)$, and we conclude $u(t,\cdot)\in\D_b$.

Part (3) follows from the sequence of identities
\[
\frac{\partial^2 u(t,x)}{\partial t^2} = \sum_{n=1}^\infty \psi_n''(t)
= \sum_{n=1}^\infty a_n \lambda^b_n e^{-\sqrt{\lambda^b_n}t} \phi^b_n(x) =
H_b u(t,x) = -\Delta u(t,x).
\]
\end{proof}

We will call the function $u(t,x) = \Pois^b_tf(x)$ the \emph{Dirichlet Poisson
integral} (respectively \emph{Neumann Poisson integral}) of $f$. 

As in the case of the heat semigroup, it is not hard to see that 
$f\mapsto\Pois^b_t f$ is a strongly continuous semigroup on $L^2(K,\mu)$, which
implies that $u(t,x)\to f(x)$ as $t\to 0$ in $L^2$. It also acts as an
approximation to the identity on continuous functions.

\begin{proposition}\label{identPois}
\begin{enumerate}
\item Let $f\in C(K)$ and $u_N(t,x)$ its Neumann Poisson integral. Then 
$||u_N(t,\cdot) - f||_{L^\infty(K,d\mu)} \to 0$ as $t\to 0$.
\item Let $f\in C(K)$ with $f|_{V_0} \equiv 0$ and $u_D(t,x)$ its
Dirichlet Poisson integral. Then 
$||u_D(t,\cdot) - f||_{L^\infty(K,d\mu)} \to 0$ as $t\to 0$.
\end{enumerate}
\end{proposition}

\begin{proof}
This proposition follows from Propositions \ref{identHeat} and \ref{subord}. 
Indeed, as
\[
\frac{t}{2\sqrt\pi} \int_0^\infty e^{-t^2/4s} \frac{ds}{s^{3/2}} =
\frac{t}{2\sqrt\pi} \int_0^\infty e^{-\frac{t^2}{4}s} s^{1/2} \frac{ds}{s} = 1,
\]
using Fubini's theorem we obtain
\[
u_b(t,x) - f(x) = \frac{t}{2\sqrt\pi} \int_0^\infty e^{-t^2/4s} 
\big( \Heat^b_s f(x) - f(x) \big)
\frac{ds}{s^{3/2}},
\]
where $\Heat^b_s$ is the heat semigroup. By Proposition \ref{identHeat}, for
any $f\in C(K)$ and any $\e > 0$ there is $\delta>0$ such that, if
$0<s<\delta$,
\[
\big|\big| \Heat^b_sf - f \big|\big|_{L^\infty(K,d\mu)} < \e,
\]
if $b=N$, and for $f\in C(K)$ with $f|_{V_0}\equiv 0$ if $b=D$. Thus
\[
\begin{split}
||u_b(t,\cdot) - f||_{L^\infty(K,d\mu)} &\le
\frac{t}{2\sqrt\pi} \int_0^\infty e^{-t^2/4s} 
\big|\big| \Heat^b_s f - f \big|\big|_{L^\infty(K,d\mu)} 
\frac{ds}{s^{3/2}}\\
&\le \frac{t}{2\sqrt\pi} \int_0^\delta e^{-t^2/4s} \e \frac{ds}{s^{3/2}} +
\frac{t}{2\sqrt\pi} \int_\delta^\infty e^{-t^2/4s} M \frac{ds}{s^{3/2}},
\end{split}
\]
where $M>0$ is such that
$\big|\big| \Heat^b_sf - f \big|\big|_{L^\infty(K,d\mu)} \le M$ uniformly in
$s$ ($f$ is continuous).

Now
\[
\frac{t}{2\sqrt\pi} \int_0^\delta e^{-t^2/4s} \frac{ds}{s^{3/2}} \le
\frac{t}{2\sqrt\pi} \int_0^\infty e^{-t^2/4s} \frac{ds}{s^{3/2}} = 1,
\]
and
\[
\frac{t}{2\sqrt\pi} \int_\delta^\infty e^{-t^2/4s} \frac{ds}{s^{3/2}} \le
\frac{t}{2\delta^{1/4}\sqrt\pi} \int_0^{1/\delta} e^{-t^2s/4} s^{1/4}
\frac{ds}{s} \le \frac{\sqrt t}{\sqrt 2\delta^{1/4}\sqrt\pi} \Gamma(1/4),
\]
so
\[
||u_b(t,\cdot) - f||_{L^\infty(K,d\mu)} \le \e + C \sqrt t
\]
for some constant $C>0$. Therefore, as $\e>0$ is arbitrary, we obtain both
cases of the proposition.
\end{proof}

We now state and prove the following theorem, which describes the boundary
behavior of $u(t,x)$ for $f\in L^p$.

\begin{theorem}\label{approx-ident}
Let $f\in L^p(K,\mu)$, $1\le p\le \infty$, and $u(t,x)$ either its Dirichlet
or
Neumann Poisson integral.
\begin{enumerate}
\item There exists a constant $A>0$ such that, for every $t>0$,
\[
|u(t,x)| \le A Mf(x),
\]
where $Mf$ is the maximal function defined in Section \ref{Maximal};
\item $u(t,\cdot) \to f$ in $L^p(K,\mu)$, if $1 \le p < \infty$; 
\item $\lim_{t\to 0}u(t,x) = f(x)$ for a.e. $x\in K$.
\end{enumerate}
\end{theorem}

For the proof of this theorem we use the following Lemma.

\begin{lemma}\label{estPois}
There exists a constant $C>0$ such that, for any $x,y\in K$ and $t>0$,
\[
P^b(t,x,y) \le C \min \Big\{ t^{-\frac{2d}{d+1}},
\frac{t}{R(x,y)^\frac{3d+1}{2}} \Big\}.
\]
\end{lemma}

\begin{proof}
This lemma follows from the estimate for the heat kernel
\[
H^b(t,x,y) \le 
A t^{-\frac{d}{d+1}}\exp \Big( -c' \big(\frac{R(x,y)^{d+1}}{t}\big)^{1/d} \Big)
\]
for some $A,c'>0$ \cite[Theorem 8.15]{Barlow98}. By Proposition \ref{subord},
\[
P^b(t,x,y) = \frac{t}{2\sqrt \pi} \int_0^\infty e^{-t^2/4s} H^b(s,x,y)
\frac{ds}{s^{3/2}},
\]
so we have
\[
\begin{split}
P^b(t,x,y) &\lesssim t \int_0^\infty e^{-t^2/4s} s^{-\frac{d}{d+1}}
\frac{ds}{s^{3/2}}
= t \int_0^\infty e^{-t^2/4s} s^{-\frac{3d+1}{2(d+1)}} \frac{ds}{s} \\
&\approx t\cdot t^{-\frac{3d+1}{d+1}} = t^{-\frac{2d}{d+1}},
\end{split}
\]
and also
\[
P^b(t,x,y) \lesssim t \int_0^\infty \exp \Big( -c'
\big(\frac{R(x,y)^{d+1}}{s}\big)^{1/d} \Big)
s^{-\frac{3d+1}{2(d+1)}} \frac{ds}{s} \lesssim
t \cdot R(x,y)^{-\frac{3d+1}{2}}.
\]
\end{proof}

\begin{proof}[Proof of Theorem \ref{approx-ident}]
\begin{enumerate}
\item This part follows by an argument analogous to the well-known Euclidean 
case: since the Poisson integral $u(t,x)$ of $f$ is given by
\[
u(t,x) = \int_K P^b(t,x,y) f(y) d\mu(y),
\]
we write
\[
|u(t,x)| \le \int_K P^b(t,x,y) |f(y)| d\mu(y) = 
\sum_{n=0}^\infty \int_{A_n(x)} P^b(t,x,y) |f(y)| d\mu(y),
\]
where $A_0(x) = \{ y\in K : R(x,y) \le t^{\frac{2}{d+1}} \}$, and
\[
A_n(x) = \{ y\in K : 2^{n-1} t^{\frac{2}{d+1}} < R(x,y) \le  
2^n t^{\frac{2}{d+1}}\}, \qquad n\ge 1. 
\]
Now, from Lemma \ref{estPois} and the estimate \eqref{scalmu},
\[
\begin{split}
\int_{A_0(x)} P^b(t,x,y) &|f(y)| d\mu(y)\\ &\lesssim t^{-\frac{2d}{d+1}} \cdot 
\frac{t^{\frac{2}{d+1} \cdot d}}{\mu(B_{t^{2/(d+1)}}(x))}
\int_{B_{t^{2/(d+1)}}(x)} |f(y)| d\mu(y)\\
&\le Mf(x).
\end{split}
\]
Similarly, for each $n\ge 1$,
\[ 
\begin{split}
\int_{A_n(x)} P^b(t,x,y) &|f(y)| d\mu(y) \\
&\lesssim t \int_{A_n(x)} \frac{1}{R(x,y)^\frac{3d+1}{2}} |f(y)| d\mu(y)\\
&\le t \cdot \frac{1}{\big( 2^{n-1}t^\frac{2}{d+1} \big)^\frac{3d+1}{2} }
\int_{B_{2^nt^{2/(d+1)}}(x)} |f(y)|d\mu(y)\\
&\lesssim 
\frac{t^{-\frac{2d}{d+1}}
2^{-n\frac{3d+1}{2}}(2^nt^{\frac{2}{d+1}})^d}{\mu(B_{2^nt^{2/(d+1)}}(x))}
\int_{B_{2^nt^{2/(d+1)}}(x)} |f(y)|d\mu(y)\\
&\le 2^{-\frac{d+1}{2} n} Mf(x).
\end{split}
\]
Therefore
\[
|u(t,x)| \lesssim \sum_{n=0}^\infty 2^{-\frac{d+1}{2} n} Mf(x) \lesssim Mf(x).
\]
\item This part follows, as in the classical case, from the fact that the 
family $\{P^b(t,x,y)\}_{t>0}$ forms an approximation to the identity for 
continuous functions, and the fact that the maximal function is weakly bounded
in $L^1$ and bounded in $L^p$, $p>1$.
\item This is standard \cite{SW}.
\end{enumerate}
\end{proof}

\begin{remark}
We observe that we can write the estimates of Lemma \ref{estPois} as
\[
P^b(t,x,y) \le \frac{C' t}{(t^2 + R(x,y)^{d+1})^\frac{3d+1}{2(d+1)}},
\]
for some $C'>0$ and either $b=N$ or $D$, and thus we have an analogous
inequality to the classical Poisson kernel.
\end{remark}

\begin{remark}\label{approx-ident-Borel}
If $\nu$ is a finite Borel measure on $K$, we can define its Dirichlet and
Neumann Poisson integrals $u(t,x) = \Pois^b_t\nu(x)$ as
\[
\Pois_t^b\nu(x) = \int_K P^b(t,x,y) d\nu(y).
\]
The same arguments as in Theorem \ref{approx-ident}(1) imply the estimate
$|u(t,x)|\lesssim M\nu(x)$, and hence $u(t,\cdot) \to \nu$ as $t\to 0$. In
particular, $||u(t,\cdot)||_{L^1}$ is uniformly bounded.
\end{remark}

\subsection{Nontangential limits}

In this section we discuss nontangential limits of Poisson integrals. We first
develop the concept of a \emph{cone} over a point $x\in K$. From the final
remark in the previous section, for $\alpha>0$ we consider the set 
\[
\Gamma_\alpha(x) = \{ (t,y)\in \R_+\times K: R(x,y)^{d+1} < \alpha t^2 \}.
\]
The set $\Gamma_\alpha(x)$ is not properly a cone; however, in the case $d>1$,
it contains the intersection of the cone 
\[
\{ (t,y)\in\R_+\times K: R(x,y) < \sqrt\alpha t \}
\]
with set $\{ (t,y)\in\R_+\times K: R(x,y)<1\}$, as $R(x,y)^2 > R(x,y)^{d+1}$ 
for such points.

We now state the following result.

\begin{theorem}\label{nontanthm}
Let $f\in L^p(K,d\mu)$, $1\le p\le \infty$, and $u(t,x)$ either its
Dirichlet or its Neumann Poisson integral. Let $\alpha>0$. Then
\begin{enumerate}
\item There exists $A_\alpha>0$ such that, for $x\in K$,
\[
\sup_{(t,y)\in\Gamma_\alpha(x)} |u(t,y)| \le A_\alpha Mf(x),
\]
where $Mf$ is the maximal function of $f$;
\item For almost every $x\in K$, 
\[
\lim_{\substack{(t,y)\to(0,x)\\(t,y)\in\Gamma_\alpha(x)}} u(t,y) = f(x).
\]
\end{enumerate}
\end{theorem}

\begin{proof}
\begin{enumerate}
\item The proof of (1) follows as the one in Theorem \ref{approx-ident}, once
we prove that, for $(t,y)\in\Gamma_\alpha(x)$, $P^b(t,y,z)$ satisfies an
estimate as in Lemma \ref{estPois} for $(t,x,z)$, \ie 
\begin{equation}\label{nontanest}
P^b(t,y,z) \le C_\alpha \min \Big\{ t^{-\frac{2d}{d+1}},
\frac{t}{R(x,z)^\frac{3d+1}{2}} \Big\},
\end{equation}
for some constant $C_\alpha >0$. 
Indeed, if $R(x,y) \le \dfrac{1}{2}R(x,z)$,
\[
R(y,z) \ge R(x,z) - R(x,y) \ge \frac{1}{2}R(x,z),
\]
and we have, by Lemma \ref{estPois},
\[
P^b(t,y,z) \le \frac{Ct}{R(y,z)^\frac{3d+1}{2}} \le
\frac{C't}{R(x,z)^\frac{3d+1}{2}}.
\]
If $R(x,y) > \dfrac{1}{2}R(x,z)$, since $(t,y)\in\Gamma_\alpha(x)$,
\[
\alpha t^2 > R(x,y)^{d+1} > \frac{1}{2^{d+1}} R(x,z)^{d+1},
\]
and therefore, by Lemma \ref{estPois}
\[
\begin{split}
P^b(t,y,z) &\le Ct^{-\frac{2d}{d+1}} = \frac{Ct}{t^\frac{3d+1}{d+1}}\\ &<
\frac{Ct}{\Big( \dfrac{1}{\alpha 2^{d+1}}(R(x,z)^{d+1}
\Big)^\frac{3d+1}{2(d+1)}} =\frac{C_\alpha t}{R(x,z)^\frac{3d+1}{2}}.
\end{split}
\]
The proof now follows as in Theorem \ref{approx-ident}, by decomposing the
Poisson integral
\[
u(t,y) = \int_K P^b(t,y,z) f(z) d\mu(z)
\]
in annuli $A_n(x)$, with center $x$, of radius $\sim 2^n t^\frac{2}{d+1}$ for
each $n$.
\end{enumerate}
\end{proof}

For the second part of the theorem we need the following lemma.

\begin{lemma}\label{int1}
For any $x\in K\setminus V_0$, 
\[
\lim_{t\to 0} \int_K P^D(t,x,y) d\mu(y) = 1.
\]
\end{lemma}

\begin{proof}
Let $U$ be a neighborhood of $V_0$ and $\delta>0$ such that 
$B_\delta(x)\cap U=\emptyset$. Let $f$ be a continuous function on $K$ such 
that $f\equiv 1$ on $K\setminus U$ and $f|_{V_0}\equiv 0$. Hence
\[
\int_K P^D(t,x,y) d\mu(y) = \int_K P^D(t,x,y)f(y)d\mu(y) + 
\int_U P^D(t,x,y)(1 - f(y))d\mu(y).
\]
By Proposition \ref{identPois},
\[
\int_K P^D(t,x,y)f(y)d\mu(y) \to f(x) = 1
\]
as $t\to 0$ and, since $R(x,y)>\delta$ for $y\in U$, 
\[
\int_U P^D(t,x,y)|1-f(y)|d\mu(y) \le A t \int_U
\frac{d\mu(y)}{R(x,y)^\frac{3d+1}{2}} < 
A_\delta \mu(U) t.
\]
Therefore 
\[
\int_K P^D(t,x,y) d\mu(y) \to 1
\]
as $t\to 0$.
\end{proof}

Observe that in the Neumann case one actually has
\[
\int_K P^N(t,x,y) d\mu(y) = 1
\]
for every $t>0$.

\begin{proof}[Proof of Theorem \ref{nontanthm}]
\begin{enumerate}
\item[(2)] Let $x\in K\setminus V_0$ be in the Lebesgue set of $f$, and let
$\e>0$. Thus, there exists $\delta>0$ such that
\begin{equation}\label{xLeb}
\frac{1}{\mu(B_r(x))} \int_{B_r(x)} |f(y) - f(x)| d\mu(y) < \e
\end{equation}
for every $r<\delta$. We thus define the function $g$ on $K$ by
\[
g(y) = 
\begin{cases}
|f(y) - f(x)| & R(x,y) < \delta\\
0 & R(x,y)\ge \delta.
\end{cases}
\]
Hence \eqref{xLeb} implies $Mg(x) < \e$. We now have
\begin{multline*}
|u(t,y) - f(x)| \le\\ \int_K P^b(t,y,z)|f(z) - f(x)| d\mu(z) + 
\Big| \int_K P^b(t,y,z) d\mu(z) - 1\Big| \cdot |f(x)|.
\end{multline*}
By Lemma \ref{int1}, the second term in the right hand side is either zero in
the Neumann case $b=N$, or goes to $0$ as $t\to 0$ in the Dirichlet case 
$b=D$. We split the first term as the sum
\[
\int_{B_\delta(x)} P^b(t,y,z)|f(z) - f(x)| d\mu(z) + 
\int_{K\setminus B_\delta(x)} P^b(t,y,z)|f(z) - f(x)| d\mu(z).
\]
Now, by part (1) of the theorem, if $(t,y)\in\Gamma_\alpha(x)$,
\[
\begin{split}
\int_{B_\delta(x)} P^b(t,y,z)|f(z) - f(x)| d\mu(z) &= \int_K
P^b(t,y,z)g(z)d\mu(z)\\ 
&\le A_\alpha Mg(x) < A_\alpha \e.
\end{split}
\]
Also, by \eqref{nontanest}, if $(t,y)\in\Gamma_\alpha(x)$,

\[
\begin{split}
\int_{K\setminus B_\delta(x)} P^b(t,y,z)|f(z) - f(x)| d\mu(z)
&\le C_\alpha t\int_{K\setminus B_\delta(x)}
\frac{|f(z)-f(x)|}{R(x,z)^\frac{3d+1}{2}} d\mu(z)\\
&\le C_{\alpha,\delta} t \int_K |f(z) - f(x)| d\mu(z)\\
&\le C_{\alpha,\delta} \big(||f||_{L^p}+|f(x)|\big) t,
\end{split}
\]
and thus goes to $0$ as $t\to 0$. Hence we have
\[
\limsup_{\substack{(t,y)\to(0,x)\\(t,y)\in\Gamma_\alpha(x)}}|u(t,y) - f(x)| \le
A_\alpha \e
\]
and, as $\e>0$ is arbitrary,
\[
\lim_{\substack{(t,y)\to(0,x)\\(t,y)\in\Gamma_\alpha(x)}} u(t,y) = f(x).
\]
Since $V_0$ is finite and the Lebesgue set of $f$ contains almost every point
of $K$ (if $f\in L^p(K,d\mu)$, then $f$ is integrable, as we have noted above),
therefore we obtain the theorem.
\end{enumerate}
\end{proof}

\section{Fatou-type theorems}\label{Fatou-sect}

\subsection{A maximum principle}

We say that a continuous function $u$ on $\R_+\times K$ is a \emph{harmonic
function} if
\begin{enumerate}
\item $u(\cdot,x)\in C^2(\R_+)$ for each $x\in K$;
\item $u(t,\cdot)\in \D$ for each $t>0$; and
\item For each $(t,x)\in\R_+\times K\setminus V_0$,
\begin{equation}\label{Lapeq}
\frac{\partial^2 u(t,x)}{\partial t^2} + \Delta u(t,x) = 0.
\end{equation}
\end{enumerate}

For example, if $u(t,x)$ is the Poisson integral of $f\in L^p(K,d\mu)$, then
$u$ is a harmonic function, by Theorem \ref{Psemi}. 

We state and prove the following result, analogous to the 
\emph{parabolic} maximum principle for solutions of the heat equation
\cite[Section 5.2]{Kigami}.

\begin{theorem}\label{maxprthm}
Let $u$ be a harmonic function on $\R_+\times K$. Then $u$ cannot take a 
maximum in $\R_+\times (K\setminus V_0)$.
\end{theorem}

As in the case of the parabolic maximum principle, the proof of Theorem
\ref{maxprthm} makes use of the following lemma, whose proof can be found in 
\cite[Lemma 5.2.4]{Kigami}.

\begin{lemma}\label{maxlem}
Let $u\in\D$. If $u(x) = \max\{u(y):y\in K\}$ for some $x\in K\setminus V_0$,
then $\Delta u(x)\le 0$.
\end{lemma}

\begin{proof}[Proof of Theorem \ref{maxprthm}]
For $n\in\Z_+$, set $u_n = u + t^2/n$. Suppose $u_n$ takes its maximum at an
interior point $(t_0,x_0)\in\R_+\times(K\setminus V_0)$. Then
\[
\frac{\partial^2 u_n}{\partial t^2}(t_0,x_0) \le 0.
\]
But
\[
\frac{\partial^2u_n}{\partial t^2} = \frac{\partial^2u}{\partial t^2} +
\frac{2}{n} = -\Delta u + \frac{2}{n},
\]
so
\[
\Delta u(t_0,x_0) = - \frac{\partial^2u_n}{\partial t^2}(t_0,x_0) +
\frac{2}{n} > 0,
\]
contradicting Lemma \ref{maxlem} because the function
\[
x\mapsto u(t_0,x) = u_n(t_0,x) - \frac{t^2_0}{n}
\]
takes its maximum at $x_0\in K\setminus V_0$.
\end{proof}

From Theorem \ref{maxprthm}, if $u$ is a bounded continuous function on
$[0,\infty)\times K$ which is harmonic on $\R_+\times K$, then, if $u$ takes 
its maximum in $(t_0,x_0)$, then $t_0=0$ or $x_0\in V_0$. 

Moreover, by taking $-u$, we can similarly conclude that its minimum, if taken
by $u$ at $(t_0,x_0)$, must satisfy $t_0=0$ or $x_0\in V_0$. We have the
following corollary.

\begin{corollary}\label{maxcor}
Let $u$ be continuous in $[0,\infty)\times K$, harmonic on $\R_+\times K$,
and $0\le a < b$. If $u(a,x)\ge 0$ and $u(b,x)\ge 0$ for every $x\in K$, and
$u(t,p) \ge 0$ for every $p\in V_0$ and $a\le t \le b$, then $u(t,x)\ge 0$ for 
every $(t,x)\in[a,b]\times K$.
\end{corollary}

We can clearly extend this corollary to any open set in $\R_+\times K$.

\begin{corollary}\label{maxcor-open}
Let $u$ be continuous in $[0,\infty)\times K$, harmonic on $\R_+\times K$ and
$\Omega\subset\R_+\times K$ an open set. If $u(t,x)\ge 0$ for
$(t,x)\in\partial\Omega$, then $u(t,x)\ge 0$ for every $(t,x)\in\Omega$.
\end{corollary}

\subsection{Harmonic functions and boundary limits}
We have observed that, if $u(t,x)$ is either the Dirichlet or Neumann Poisson
integral of $f\in L^p(K,d\mu)$, then $u$ is a harmonic function on $\R_+\times
K$. Moreover, Theorem \ref{approx-ident} implies that 
\[
\sup_{t>0} ||u(t,\cdot)||_{L^p(K,d\mu)} < \infty
\]
if $1\le p\le\infty$. We now prove the converse for uniformly bounded Dirichlet
harmonic functions, a result analogous to the classical Fatou's theorem 
\cite{ABR}.

One can easily observe that if $u(t,x)$ is the Dirichlet Poisson integral of a
function on $K$, then $u(t,p)=0$ for every $p\in V_0$ and $t>0$, as the
Dirichlet Poisson kernel satisfies $P^D(t,p,y) = 0$ for $p\in V_0$.

We say that a continuous function $u$ on $\R_+\times K$ is a \emph{Dirichlet
harmonic function} if $u$ is harmonic and $u(t,p) = 0$ for every $t>0$ and
$p\in V_0$.

\begin{theorem}\label{Fatou}
Let $u$ be a Dirichlet harmonic function on $\R_+\times K$ such that
\[
\sup_{t>0} ||u(t,\cdot)||_{L^\infty(K,d\mu)} < \infty.
\]
Then $u$ is the Dirichlet Poisson integral of a function
$f\in L^\infty(K,d\mu)$.
\end{theorem}

\begin{proof}
Let $u$ be a Dirichlet bounded harmonic function on $\R_+\times K$, and $M>0$
such that $|u(t,x)|\le M$ for every $(t,x)\in\R_+\times K$. 

For $x\in K$ and $n\in\Z_+$, set $f_n(x) = u(1/n,x)$, and let $u_n(t,x)$ be 
the Dirichlet Poisson integral of $f_n$. We define then, for
$(t,x)\in\R_+\times K$,
\[
U_n(t,x) = u\big(t+\frac{1}{n},x\big) - u_n(t,x).
\]
We claim that $U_n(t,x)\equiv 0$. First note that $U_n$ is bounded, since
$|u|\le M$ and, by Theorem \ref{approx-ident}(1),
\[
|u_n(t,x)| \le A Mf_n(x) \le A||Mf_n||_{L^\infty(K)} \le A||f_n||_{L^\infty(K)}
\le AM.
\]
Thus $|U_n(t,x)|\le (A+1)M = A'M$. Moreover, since $f_n$ is continuous,
$u_n(t,x)$ can be extended to $t=0$ with $u_n(0,x)=f_n(x)$, and hence
\[
U_n(0,x) = u\big(\frac{1}{n},x\big)-f_n(x) = 0
\]
for every $x\in K$. Fix $(t_0,x_0)\in\R_+\times K$ and $\e>0$ such that
$\dfrac{1}{\e}>t_0$. Define, for $(t,x)\in[0,\infty)\times K$,
\[
U(t,x) = U_n(t,x) + A'M\e t.
\]
Then $U$ is continuous in $[0,\infty)\times K$ and harmonic on $\R_+\times K$.
Moreover, for $t=0$, $U(0,x) = 0$ for every $x\in K$ and, for
$t=\dfrac{1}{\e}$,
\[
U\big(\frac{1}{\e},x\big) = U_n\big(\frac{1}{\e},x\big) + A'M \ge 0,
\]
since $|U_n(t,x)|\le A'M$. Finally, as $U_n(t,p)=0$ for $p\in V_0$, we have
\[
U(t,p) = A'Mt\e \ge 0
\]
for $p\in V_0$. Therefore, by Corollary \ref{maxcor}, $U(t,x)\ge 0$, and hence
\[
U_n(t_0,x_0) \ge -A'M\e t_0.
\]
Since $\e$ is arbitrary, $U_n(t_0,x_0) \ge 0$.

Similarly, taking $-U_n(t,x)$, we can conclude that $U_n(t_0,x_0)\le 0$, and
therefore $U_n(t_0,x_0)=0$ for any $(t_0,x_0)\in[0,\infty)\times K$.

This shows that, for any $n$, $u(t+1/n,x)$ is the Dirichlet Poisson integral of
$f_n$, \ie
\[
u\big(t+\frac{1}{n},x\big) = \int_K P^D(t,x,y) f_n(y) d\mu(y).
\]
Now, for every $n$, $||f_n||_{L^\infty} \le M$, so by the weak-$*$ compactness 
of the ball in $L^\infty(K,d\mu)$, there is a subsequence $f_{n_k} \to f$ 
weakly in $L^\infty(K,d\mu)$, and hence
\[
\int_K \psi(y) f_{n_k}(y) d\mu(y) \to \int_K \psi(y) f(y) d\mu(y)
\]
for any $\psi\in L^1(K,d\mu)$. Taking, for each $t>0$ and $x\in K$, 
$\psi(y) = P^D(t,x,y)$, we obtain, by the continuity of $u(t,x)$, that
\[
u(t,x) = \int_K P^D(t,x,y) f(y) d\mu(y).
\]
Therefore $u(t,x)$ is the Dirichlet Poisson integral of the $L^\infty$ function
$f$.
\end{proof}

So we finally have, from Theorems \ref{nontanthm} and \ref{Fatou}, the
following corollary.

\begin{corollary}\label{Fatou-nontan}
Then a bounded Dirichlet harmonic function on $\R_+\times K$ has nontangential
limit at $x$ as $t\to 0$ for almost every $x\in K$.
\end{corollary}

As in the classical setting \cite[Chapter VII]{Stein70}, Theorem \ref{Fatou} 
can be extended to the case $1\le p < \infty$. Together with Theorem
\ref{approx-ident} and Remark \ref{approx-ident-Borel}, we obtain the following
Corollary.

\begin{corollary}\label{Fatou-nontan-p}
Suppose $u$ is a Dirichlet harmonic function on $\R_+\times K$. If,
$1< p\le\infty$, then $u$ is the Poisson integral of some $f\in L^p$ if and
only if
\[
\sup_{t>0} ||u(t,\cdot)||_{L^p} < \infty.
\]
Moreover, $u$ is the Poisson integral of some finite Borel measure on $K$ if
and only if
\[
\sup_{t>0} ||u(t,\cdot)||_{L^1} < \infty.
\]
\end{corollary}

\begin{proof}
If $u$ is the Dirichlet Poisson integral of either some $f\in L^p$ or a finite
Borel measure on $K$, the estimates follow from Theorem \ref{approx-ident} and
Remark \ref{approx-ident-Borel}.

Now, suppose $p<\infty$ and $\sup_{t>0} ||u(t,\cdot)||_{L^p} < \infty$. Since
$\{\Pois^D_t\}_{t>0}$ is a semigroup, for any 
$t > t_0 > 0$,
\[
|u(t,x)| \le \int_K P^D(t-t_0,x,y)|u(t_0,y)| d\mu(y) \lesssim
\Big( \int_K \big| P^D(t-t_0,x,y) \big|^q d\mu(y) \Big)^{1/q},
\]
where $q$ is the conjugate exponent to $p$, independently of $t_0$ because
$||u(t,\cdot)||_{L^p}$ is uniformly bounded in $t>0$. From Lemma 
\ref{estPois}, and the same decomposition as in the proof of Theorem
\ref{approx-ident}, one proves 
\[
\Big( \int_K \big| P^D(t-t_0,x,y) \big|^q d\mu(y) \Big)^{1/q} \lesssim
(t - t_0)^{-\frac{2d}{d+1} \cdot \frac{1}{p}},
\]
uniformly in $t$, $t_0$ and $x$, so we have 
$u(t,x) \lesssim t^{-\frac{2d}{d+1} \cdot \frac{1}{p}}$ uniformly in $x$.
By Theorem \ref{Fatou}, $u(t + 1/k,x) = \Pois^D f_k(x)$, where
$f_k(x) = u(1/k,x)$. Since, by assumption, the $L^p$ norms of $f_k$ are
uniformly bounded, the Corollary follows from a similar weak-$*$ argument as
in the proof of Theorem \ref{Fatou}.
\end{proof}

\section{A local Fatou theorem}\label{local-sect}

In this section we prove a local Fatou theorem, analogous to the classical
nontangential convergence at the boundary of \emph{nontangentially bounded}
harmonic functions (\cite[Thm. VII.3]{Stein70}, \cite[Thm. 7.30]{ABR}).

As we'll need estimates from below for the Neumann Poisson kernel, we can only
prove these for the so called \emph{nested} fractals \cite{Lindstrom}, 
which we define below.

\subsection{Nested fractals}

We begin by defining the concept of affine nested fractals. We now assume that
$(K,S,\{F_i\}_{i\in S})$ is a connected post-critically finite self-similar
set, $K\subset \R^n$ and each $F_i$ is the restriction to $\R^n$ of a
similitude $F_i:\R^n\to\R^n$, that is, a map of the form $x\mapsto cUx + a$,
where $0 < c < 1$, $a\in\R^n$ and $U\in O(n)$.

We say that a homeomorfism $f:K\to K$ is a \emph{symmetry} of 
$(K,S,\{F_i\}_{i\in S})$ if, for every $m\ge 0$, there exists a map 
$f_m:W_m\to W_m$ such that
\[
f(F_w(V_0)) = F_{f_m(w)}(V_0), \qquad\text{for every $w\in W_m$.}
\]
That is, a symmetry preserves the self-similar structure of $K$.

For $x,y\in\R^n$, let $H_{xy}$ be the bisecting hyperplane of the segment from
$x$ to $y$, and $\psi_{xy}:\R^n\to\R^n$ be the reflection with respect to
$H_{xy}$.

\begin{definition}
We say that $(K,S,\{F_i\}_{i\in S})$ is an \emph{affine nested fractal} if
$\psi_{xy}|_K$ is a symmetry for any $x,y\in V_0$.
\end{definition}

In other words, reflecting $K$ by each pair of points in its boundary $V_0$
preserves its self-similar structure. Examples of affine nested fractals are
the Sierpinski gasket, the Vicsek set or the pentakun, among others (see
\cite[Section 3.8]{Kigami} for more examples and a through discussion of affine
nested fractals).

It can be proven \cite[Section 3.8]{Kigami} that if $(K,S,\{F_i\}_{i\in S})$
is an affine nested fractal, then there exist a harmonic structure with all 
$r_i$ equal to each other, say $r_i = r$, $i\in S$. We assume this harmonic
structure is regular, \ie $r < 1$.

The main result of interest to us is the following theorem \cite{FHK}.

\begin{theorem}\label{affine-est}
If $H^N$ is the Neumann heat kernel on the affine nested fractal $K$, there
exist positive constants $c_1, c_2, c_3, c_4$ such that
\begin{multline}
c_1 t^{-d/(d+1)}  \exp \Big( -c_2 \Big(\frac{R(x,y)^{d+1}}{t}\big)^{1/(d_w-1)} 
\Big) \le H^N(t,x,y) \\ \le
c_3 t^{-d/(d+1)}  \exp \Big( -c_4 \Big(\frac{R(x,y)^{d+1}}{t}\big)^{1/(d_w-1)}
\Big),
\end{multline}
where $d_w$ is the \emph{walk dimension} with respect to \emph{shortest path}
metric.
\end{theorem}

A discussion on the shortest path metric and walk dimension can be found in 
\cite{Barlow98}. From Theorem \ref{affine-est} we obtain the following
integral estimate for the Neumann Poisson kernel.

\begin{corollary}\label{affine-Pest}
Let $P^N$ be the Neumann Poisson kernel on the affine nested fractal $K$, and
$\alpha>0$. Then there exists $c_\alpha>0$ such that 
\[
\int_B P^N(t,x,y) d\mu(y) \ge c_\alpha,
\]
where $B = B_{(\alpha t^2)^{1/(d+1)}}(x)$ is the ball of radius 
$(\alpha t^2)^\frac{1}{d+1}$ with center in $x$.
\end{corollary}

\begin{proof}
From Theorem \ref{affine-est} we have
\[
\begin{split}
P^N(t,x,y) &= 
\frac{t}{2\sqrt{\pi}} \int_0^\infty e^{-t^2/4s} H^N(s,x,y) \frac{ds}{s^{3/2}} 
\\ &\gtrsim 
t \int_0^\infty e^{-t^2/4s} s^{-\frac{3d+1}{2(d+1)}} 
e^{-c \big(\frac{R(x,y)^{d+1}}{s}\big)^{1/(d_w-1)}} \frac{ds}{s}.
\end{split}
\]
Thus
\[
\begin{split}
\int_B P^N(t,x,y) d\mu(y) &\gtrsim
\int_B t \int_0^\infty e^{-t^2/4s} s^{-\frac{3d+1}{2(d+1)}} 
e^{-c \big(\frac{R(x,y)^{d+1}}{s}\big)^{1/(d_w-1)}} \frac{ds}{s} d\mu(y)\\
&\ge 
t \int_B \int_0^\infty e^{-t^2/4s} s^{-\frac{3d+1}{2(d+1)}} 
e^{-c (\frac{\alpha t^2}{s})^{1/(d_w-1)}} \frac{ds}{s} d\mu(y),
\end{split}
\]
since $R(x,y)^{d+1} < \alpha t^2$ for $y\in B$.

Now, after the change of variables $s\mapsto \alpha t^2 s$, we finally obtain
\[
\begin{split}
\int_B P^N(t,x,y) &d\mu(y)\\ &\gtrsim
t \cdot \alpha^{-\frac{3d+1}{2(d+1)}} t^{-\frac{3d+1}{d+1}} 
\int_B \int_0^\infty e^{-1/4\alpha s} s^{-\frac{3d+1}{2(d+1)}} 
e^{-c/s^{1/(d_w-1)}} \frac{ds}{s} d\mu(y)\\
&= \alpha^{-\frac{3d+1}{2(d+1)}} t^{-\frac{2d}{d+1}} 
\int_0^\infty e^{-1/4\alpha s} s^{-\frac{3d+1}{2(d+1)}} 
e^{-c/s^{1/(d_w-1)}} \frac{ds}{s} \int_B d\mu(y)\\ &\ge c_\alpha > 0,
\end{split}
\]
because the integral in $s$ converges (and its positive) and 
\[
\int_B d\mu(y) = \mu\big(B_{(\alpha t^2)^{1/(d+1)}}(x)\big) \sim (\alpha
t^2)^\frac{d}{d+1},
\]
by \eqref{scalmu}.
\end{proof}

Note that, in fact, $c_\alpha = c\, \alpha^{1/2}$. Corollary \ref{affine-Pest}
allows us to construct a harmonic function, with nontangential limit zero,
bounded away from zero at the boundary of a union of truncated cones. We define
a \emph{truncated cone} on $\R_+\times K$, for $h,\alpha > 0$ and $x\in K$, as
the set
\[
\Gamma_\alpha^h(x) = 
\{ (t,y)\in\R_+\times K: R(x,y)^{d+1} < \alpha t^2,\, 0 < y < h \}.
\]

\begin{lemma}\label{harmonic-min}
Let $E\subset K$ be a measurable set, $\alpha > 0$, and $\displaystyle
\Omega = \bigcup_{x\in E} \Gamma_\alpha^1(x)$.
Then there exists a positive harmonic function $v$ on $\R_+\times K$ such
that
\begin{enumerate}
\item $v\ge 1$ on $(\partial\Omega) \cap (\R_+\times K)$; and
\item $v$ has nontangential limit $0$ at almost every point of $E$.
\end{enumerate}
\end{lemma}

\begin{proof}
Let $P^N$ be the Neumann Poisson kernel on $K$ and define the 
function $w$ on $\R_+\times K$ by
\[
w(t,x) = \int_K P^N(t,x,y)\chi_{K\setminus E}(y) d\mu(y) + t,
\]
where $\chi_{K\setminus E}$ is the characteristic function of $K\setminus E$.

We see that $w\ge 0$ and, by Theorem \ref{nontanthm}, $w$ has nontangential 
limit $0$ at almost every point of $E$, so we need to verify that there exists
$\delta>0$ such that $w\ge \delta$ on $(\partial\Omega) \cap (\R_+\times K)$.
Clearly, $w(1, x) \ge 1$ for every $x\in K$. 

Now, we observe that $(t,y) \in \Gamma_\alpha(x)$ if and only if 
$R(x,y) < (\alpha t^2)^\frac{1}{d+1}$, that is 
$x\in B = B_{(\alpha t^2)^{1/(d+1)}}(y)$, the ball of radius 
$(\alpha t^2)^\frac{1}{d+1}$with center $y$. Hence, if 
$(t,y)\in\partial\Omega$, $x\not\in B$ for every $x\in E$, and thus
$B\subset K\setminus E$.

We thus obtain, by Corollary \ref{affine-Pest}, 
\[
\begin{split}
\int_K P^N(t,x,y)\chi_{K\setminus E}(y) d\mu(y) &=
\int_{K\setminus E} P^N(t,x,y) d\mu(y) \ge \int_B P^N(t,x,y) d\mu(y)\\
&\ge c_\alpha > 0.
\end{split}
\]
Therefore, if we choose $M=\max\{1/c_\alpha, 1\}$, the function $v = Mw$
satisfies the properties required.
\end{proof}

\subsection{Nontangentially bounded functions}

\begin{definition}
Let $u$ be a function on $\R_+\times K$. We say that $u$ is 
\emph{nontangentially bounded} at $x\in K$ if $u$ is bounded on some
$\Gamma_\alpha^h(x)$.
\end{definition}

Note that this definition involves only one truncated cone for each $x$, while
the definition of nontangential limit involves all cones over $x$, regardless
of their apperture.

We also observe that, in the case where $u$ is continuous on $\R_+\times K$, 
then $u$ is nontangentially bounded at $x$ if and only if is bounded in some
$\Gamma_\alpha^1(x)$.

We now prove the following theorem, analog to the classical local Fatou Theorem
\cite{ABR}. Remember that we assume that $K$ is an affine nested fractal.

\begin{theorem}\label{localFatou}
Let $u$ be harmonic on $\R_+\times K$ and nontangentially bounded at each
point in the set $E\subset K$. Then $u$ has a nontangential limit at almost
every point of $E$.
\end{theorem}

\begin{proof}
Analogously to the classical case,  we prove this theorem in a sequence of
steps. For each positive integer $k$, define the set
\[
E_k = \{ x\in K: |u(t,y)|\le k \text{ for } (t,y)\in\Gamma_{1/k}^1(x) \}.
\]
As $u$ is continuous on $\R_+\times K$, each $E_k$ is closed and 
$E = \bigcup_k E_k$.

\begin{step}
$u$ is bounded on $\Gamma_\alpha^1(x)$ for every $\alpha>0$ and for almost
every $x\in E_k$.
\end{step}

Observe that, if $x\in E_k$ is in the Lebesgue set of $\chi_{E_k}$, then
\[
\lim_{r\to 0} \frac{\mu(B_r(x)\cap E_k)}{\mu(B_r(x))} = 1.
\]
Recall also that there exist constants $A_1,A_2>0$ such that, for sufficiently 
small $r>0$ (say, $r < \bar r$), 
\[
A_1 r^d \le \mu(B_r(x)) \le A_2 r^d.
\]

Let $x\in E_k$ be in the Lebesgue set of $\chi_{E_k}$. As $u$ is continuous on
$\R_+\times K$, it is sufficient to prove that, for each $\alpha > 0$, there
exists $h>0$ such that 
$\Gamma_\alpha^h(x) \subset \bigcup_{z\in E_k}\Gamma_{1/k}^1(z)$. 
It is of course sufficient to consider the case $\alpha > 1/k$.

Let $\delta > 0$ such that $\delta < \min\{\bar r, k^{-1/(d+1)} \}$ and, for
$0 < r \le \delta$, 
\begin{equation}\label{density-est}
\frac{\mu(B_r(x)\cap E_k)}{\mu(B_r(x))} > 1 - \frac{A_1}{A_2} \Big(
\frac{k^{-1/(d+1)}}{\alpha^{1/(d+1)} + k^{-1/(d+1)}} \Big)^d.
\end{equation}
Set $h = \Big( \dfrac{\delta}{2\alpha^{1/(d+1)}} \Big)^{(d+1)/2}$. Thus $h<1$
and we  shall prove $\Gamma_\alpha^h(x) \subset \Omega$.

We first observe that, if $(t,y)\in\Gamma_\alpha^h(x)$, then
\[
B_{(t^2/k)^{1/(d+1)}}(y) \cap E \not= \emptyset.
\]
Indeed, if $(t,y)\in\Gamma_\alpha^h(x)$, then $R(x,y) < (\alpha t^2)^{1/(d+1)}$
and $0 < t < h$. Now, for $z\in B_{(t^2/k)^{1/(d+1)}}(y)$, 
$R(y,z) < (t^2/k)^{1/(d+1)}$ and hence
\[
R(x,z) < (\alpha^{1/(d+1)} + k^{-1/(d+1)}) t^{2/(d+1)},
\]
so $z\in B_r(y)$, where $r = (\alpha^{1/(d+1)} + k^{-1/(d+1)}) t^{2/(d+1)}$.
Therefore
\[
B_{(t^2/k)^{1/(d+1)}}(y) \subset B_r(y).
\]

If $B_{(t^2/k)^{1/(d+1)}}(y) \cap E_k = \emptyset$, then
$B_r(x)\cap E_k \subset B_r(x) \setminus B_{(t^2/k)^{1/(d+1)}}(y)$ and thus
\[
\begin{split}
\frac{\mu(B_r(x)\cap E_k)}{\mu(B_r(x))} &\le 
\frac{\mu\big(B_r(x) \setminus B_{(t^2/k)^{1/(d+1)}}(y)\big)}{\mu(B_r(x))}
= 1 - \frac{\mu(B_{(t^2/k)^{1/(d+1)}}(y))}{\mu(B_r(x))}\\
&\le 1 - \frac{A_1 \big( (t^2/k)^{1/(d+1)} \big)^d}{A_2 r^d} =
1 - \frac{A_1}{A_2} \Big( \frac{k^{-1/(d+1)}}{\alpha^{1/(d+1)} +
k^{-1/(d+1)}} \Big)^d,
\end{split}
\]
which contradicts \eqref{density-est} since
\[
r = (\alpha^{1/(d+1)} + k^{-1/(d+1)}) t^{2/(d+1)} <
(\alpha^{1/(d+1)} + k^{-1/(d+1)}) h^{2/(d+1)} < \delta,
\]
by the choice of $h$. 

Hence there exists $x_0\in B_{(t^2/k)^{1/(d+1)}}(y) \cap E_k$, which implies
$(t,y)\in\Gamma_{1/k}^1(x_0)$, as desired. This finishes the proof of Step 1.

Therefore, there is a subset $F\subset E$ such that $\mu(E\setminus F) = 0$ and
$u$ is bounded in every cone $\Gamma_\alpha^1(x)$ for $\alpha>0$ and $x\in F$.

In particular, for a fixed $\alpha>0$, $u$ is bounded in every
$\Gamma_\alpha^1(x)$ for every $x\in F$. We can thus write $F = \bigcup_k F_k$,
with
\[
F_k = \{ x\in F: |u(t,y)|\le k \text{ for } (t,y)\in\Gamma_\alpha^1(x) \}.
\]

\begin{step}
At almost every $x\in F_k$, the limit
\[
\lim_{\substack{(t,y)\to(0,x)\\(t,y)\in\Gamma_\alpha(x)}} u(t,y)
\]
exists. That is, $u$ has a limit at the boundary point $x$ within the cone
$\Gamma_\alpha(x)$.
\end{step}

Let $\displaystyle\Omega = \bigcup_{x\in F_k} \Gamma_\alpha^1(x)$. As $u$ is
continuous on $\R_+\times K$, we may assume, say, $|u|\le 1$ on the set
$\displaystyle\Omega' = \bigcup_{x\in F_k} \Gamma_\alpha^2(x)$. Without loss of
generality, we can also assume $u$ is real valued.

Now, for each $n\ge 1$, let 
\[
G_n = \{ y\in K: \text{there exists $x\in F_k$ such that 
$R(x,y)^{d+1} < \alpha/n^2$} \}.
\]
Each $G_n$ is open and $G_n \subset F_k$. Now define the functions $f_n$ on $K$
by
\[
f_n(x) = \chi_{G_n}(x) u(1/n,x).
\]
As $(1/n,x)\in\Omega$ if and only if $x\in G_n$, we have $|f_n|\le 1$, and 
thus, passing to a subsequence, $(f_n)$ weakly-$*$ converges to some 
$f\in L^\infty(K,d\mu)$, in particular 
\[
\mathcal P_t^N f_n(x) \to \mathcal P_t^N f(x)
\]
for each $(t,x)\in\R_+\times K$, where $\mathcal P_t^N f_n$ and 
$\mathcal P_t^N f$ are the Neumann Poisson integrals of the functions $f_n$ and
$f$, respectively. Moreover, the function 
\[
u_n(t,x) = \mathcal P_t^N f_n(x) - u(t + 1/n, x)
\]
is harmonic and extends continuously to $\{0\}\times G_n$, because $f_n$ is
continuous, with $u_n(0,x)=0$ for each $x\in G_n$. Moreover, $|u_n|\le 2$ on 
the closure of $\Omega$, since $t+1/n \le 2$.

If we choose $v$ as in Lemma \ref{harmonic-min}, then 
\[
\liminf_{(t,y)\to\partial\Omega} (2v \pm u_n)(t,x) \ge 0,
\]
so, by the maximum principle (Corollary \ref{maxcor-open}), $2v \pm u_n \ge 0$
on $\Omega$. Taking $n\to\infty$, we conclude that
\[
|\mathcal P_t^N f(x) \pm u| \le 2v
\]
on $\Omega$ and, as $\mathcal P_t^N f(x)$ and $v$ have nontangential limits
(Theorem \ref{nontanthm} and Lemma \ref{harmonic-min}), we obtain Step 2.


Hence, for each $\alpha>0$, $u$ has a limit at the boundary, within
the cone $\Gamma_\alpha(x)$, for every point $x$ in a subset
$F'_\alpha\subset F$ with $\mu(F\setminus F'_\alpha) = 0$. 

If we take $F' = \bigcup_k F'_k$, we conclude that $u$ has nontangential limit
at
every point in $F'$, with $\mu(E\setminus F')=0$, as desired.
\end{proof}


\section*{Acknowledgements}

The author would like to thank the referee for providing useful suggestions,
which led to the improvement of some of the results of this paper.


\end{document}